\documentclass[3p]{elsarticle}

\usepackage{url}
\usepackage{amssymb,amsfonts,amsmath}
\usepackage{euler}
\usepackage[all]{xy}
\UseComputerModernTips

\newcommand{\diagram}{\xymatrix} 

\newcommand{\arrowlength}{3.5}
\newcommand{\medarrowlength}{5} 

\newcommand{\myarrow}[3] 
   {\mathrel{\xy\ar(#1,0)^-{#2}_-{#3}\endxy}}
\renewcommand{\rightarrow}
   {\myarrow{\arrowlength}{}{}}

\renewcommand{\leftarrow}
   {\myarrow{-\arrowlength}{}{}}
\newcommand{\stackrightarrow}[1]
   {\myarrow{\arrowlength}{#1}{}}

\newcommand{\myArrow}[3] 
   {\mathrel{\xy\ar@{=>}(#1,0)^-{#2}_-{#3}\endxy}}
\renewcommand{\Rightarrow}
   {\myArrow{\arrowlength}{}{}}
\renewcommand{\Leftarrow}
   {\myArrow{-\arrowlength}{}{}}

\newcommand{\myEq}[1] 
   {\mathrel{\xy\ar@{=}(#1,0)\endxy}}

\newcommand{\mymono}[2] 
   {\mathrel{\,\;\xy\ar@{>->}(#1,0)^-{#2}\endxy}}
\newcommand{\rightmono}
   {\mymono{\arrowlength}{}}

\newcommand{\myepi}[3] 
   {\mathrel{\,\;\xy\ar@{->>}(#1,0)^-{#2}_-{#3}\endxy}}
\newcommand{\rightepi}
   {\myepi{\medarrowlength}{}{}}

\newcommand{\leftepi}
   {\myepi{-\medarrowlength}{}{}}

\newcommand{\myembedding}[2] 
   {\mathrel{\:\xy\ar@{^(->}(#1,0)^-{#2}\endxy}}
\newcommand{\rightembedding}
   {\myembedding{\arrowlength}{}}

\newcommand{\verticalarrow}[3] 
{\xy (0,4)*+{\mbox{\scriptsize$#1$}}
     \ar(0,-4)*+{\mbox{\scriptsize$#3$}} _-{\mbox{$#2$}} 
 \endxy}

\newcommand{\mymapsto}[2] 
  {\mathrel{\xy\ar@{|->}(#1,0)^-{#2}\endxy}}
\renewcommand{\mapsto}
  {\mymapsto{\arrowlength}{}}

\newcommand{\mydasharrow}[3] 
   {\mathrel{\xy\ar@{.>}(#1,0)^-{#2}_-{#3}\endxy}}

\newcommand{\adjunction}[2] 
{\diagram{\ar[r]<-1ex>_-{#1} \ar@0[r]|-\top& \ar[l]<-1ex>_-{#2}}}

\newcommand{\reflection}[2] 
{\diagram@C=15pt{\ar@{^(->}[r]<-1ex>_-{#1} \ar@0[r]|-\bot& \ar[l]<-1ex>_-{#2}}}

\newcommand{\coreflection}[2] 
{\diagram{\ar@{^(->}[r]<-1ex>_-{#1} \ar@0[r]|-\top& \ar[l]<-1ex>_-{#2}}}

\newcommand{\fib}[3] 
{\xy (0,4)*+{\mbox{\scriptsize$#1$}} 
     \ar(0,-4)*+{\mbox{\scriptsize$#3$}} _-{\mbox{$#2$}} 
 \endxy}
\newcommand{\fibmapsto}[3] 
{\xy (0,4)*+{\mbox{\scriptsize$#1$}} 
 \ar@{|.>}(0,-4)*+{\mbox{\scriptsize$#3$}} _-{\mbox{$#2$}} 
 \endxy}

\newcommand{\leftparallelarrow} 
  {\diagram@C=15pt{& \ar[l]<-1ex> \ar[l]<1ex>}}
\newcommand{\stackleftparallelarrow}[2]              
  {\diagram{& \ar[l]<-1ex>_-{#1} \ar[l]<1ex>^-{#2}}} 
\newcommand{\stackleftrightarrow}[2]                 
  {\diagram{\ar[r]<0.5ex>^-{#1} & \ar[l]<0.5ex>^-{#2}}} 

\newcommand{\natrightarrow}
  {\stackrightarrow{.}}

\newcommand{\shortarrowlength}{3} 

\newcommand{\halfarrowlength}{2} 

\newcommand{\relrightarrow}
{\mathrel{\xy\ar@{-}(\halfarrowlength.5,0)\endxy\xy\ar@{|->}(\shortarrowlength.5,0)\endxy}}

\newtheorem{theorem}{Theorem}[section]
\newtheorem{corollary}[theorem]{Corollary}
\newtheorem{lemma}[theorem]{Lemma}
\newtheorem{proposition}[theorem]{Proposition}

\newdefinition{definition}[theorem]{Definition}
\newdefinition{remark}[theorem]{Remark}

\newproof{proof}{Proof}
\newproof{notation}{\textit{Notation}}

\newenvironment{myitemize}
  {\begin{list}{$\bullet$}
  {\setlength{\topsep}{2pt}
   \setlength{\partopsep}{2pt}
   \setlength{\itemsep}{2.5pt}
   \setlength{\parsep}{2.5pt}
   \setlength{\leftmargin}{1em}
   \setlength{\labelwidth}{.5em}}}
  {\end{list}}

\newcommand{\hide}[1]{}

\newcommand{\ie}{i.e.}
\newcommand{\cf}{cf.}

\newcommand{\Def}[1]{{\em #1\/}}
\newcommand{\Rem}[1]{{\em #1\/}}
\newcommand{\noRem}[1]{#1}

\newcommand{\eqdef}{\stackrel{\mathrm{def}}{=}}

\newcommand{\fstarg}{\mbox{$\_\!\_\,$}}
\newcommand{\sndarg}{\mbox{$\fstarg\hspace{-2.53mm}\raisebox{.5mm}{\fstarg}$}}

\newcommand{\scatfont}[1]{\mathbf{#1}}
\newcommand{\lscatfont}[1]{\mbox{\boldmath$\mathscr{#1}$}}
\newcommand{\bicatfont}[1]{\mbox{\boldmath$\mathcal{#1}$}}

\newcommand{\CAT}{\lscatfont{C\hspace{-0.6mm}A\hspace{-0.5mm}T}}
 
\newcommand{\AF}{\bicatfont{A\hspace{-.5mm}F}}
\newcommand{\gAF}{\AF}
\newcommand{\Cat}{\lscatfont{C}\hspace{-.2mm}\mbox{\boldmath$\mathit{at}$}}
\newcommand{\Set}{\lscatfont{S}\hspace{-.2mm}\mbox{\boldmath$\mathit{et}$}}
\newcommand{\Bij}{\scatfont{P}}
\newcommand{\Scat}{\mbox{\boldmath$\Sigma$}} 

\newcommand{\iso}{\cong}
\newcommand{\dcomp}{\cdot}
\newcommand{\id}{\mathrm{id}}
\renewcommand{\hom}[1]{[#1]}
\newcommand{\bighom}[1]{\big[#1\big]}
\newcommand{\op}{\circ}
\renewcommand{\natrightarrow}{\Rightarrow}
\newcommand{\Sub}{\mathrm{Sub}}

\newcommand{\scat}[1]{{\mathbb{#1}}}
\newcommand{\scatA}{\scat{A}}
\newcommand{\scatB}{\scat{B}}
\newcommand{\scatC}{\scat{C}}

\newcommand{\gpdG}{\scat{G}}
\newcommand{\gpdH}{\scat{H}}

\newcommand{\lscat}[1]{{\mathcal{#1}}}

\newcommand{\pair}[1]{{\langle{#1}\rangle}}
\newcommand{\bigpair}[1]{{\big\langle{#1}\big\rangle}}
\newcommand{\copair}[1]{{[{#1}]}}
\newcommand{\inj}[1]{\amalg_{#1}}

\newcommand{\conerightarrow}{\stackrightarrow{.}}

\newcommand{\Nat}{\mathbb{N}}
\newcommand{\setof}[1]{{\{{#1}\}}}
\newcommand{\bigsetof}[1]{{\big\{{#1}\big\}}}
\newcommand{\Bigsetof}[1]{{\Big\{{#1}\Big\}}}
\newcommand{\suchthat}{\mid}

\newcommand{\yon}{{\mathrm{y}}}
\newcommand{\PSh}[1]{{\widehat{#1}}}
\newcommand{\act}{\cdot}

\newcommand{\ntensor}[1] 
  {\mathrel{\otimes\hspace{-2.5mm}_{\raisebox{-1.35mm}{\tiny$#1$}}}}

\newcommand{\freesmcname}{!}
\newcommand{\freesmc}[1]{{\freesmcname{#1}}}
\newcommand{\listof}[1]  
{\mbox{$\langle\hspace{-.85mm}\raisebox{-.3mm}{\large$[$}$}{#1}\mbox{$\raisebox{-.3mm}{\large$]$}\hspace{-.85mm}\rangle$}}

\newcommand{\scriptlistof}[1]
{\mbox{\scriptsize$\langle\hspace{-.75mm}\raisebox{-.35mm}{\small$[$}$}{#1}\mbox{\scriptsize$\raisebox{-.35mm}{\small$]$}\hspace{-.625mm}\rangle$}}

\newcommand{\card}[1]{{|#1|}}
\newcommand{\indmap}[1]{{\underline{#1}}}
\newcommand{\objmap}[1]{{\overline{#1}}}

\newcommand{\Sf}{{\mathrm{S}}}
\newcommand{\gen}{^\circ}

\newcommand{\lan}[1]{\widetilde{#1}}
\newcommand{\coend}{\int}

\newcommand{\widewidehat}[1]{\widehat{#1}\hspace{-2mm}\raisebox{.5mm}{$\widehat{\hspace{2mm}}$}}

\newcommand{\minipagelength}
{14cm}

\begin{document}

\title{Analytic functors between presheaf categories over
  groupoids\tnoteref{writeupnote}} 
\tnotetext[writeupnote]{A write-up of~\cite{Fiore}.}

\author{Marcelo Fiore\fnref{supportfn}}
\ead{Marcelo.Fiore@cl.cam.ac.uk}
\address{University of Cambridge, Computer Laboratory,\\ 
  15 JJ Thomson Avenue, Cambridge CB3 0FD, UK}
\fntext[supportfn]{Research supported by an EPSRC Advanced Research Fellowship
  (2000--2005).}

\author{\bigskip\bigskip\it
Dedicated to Glynn Winskel on the occasion of his 60$^\mathit{th}$ birthday
}

\begin{abstract}
The paper studies analytic functors between presheaf categories. 
%
Generalising results of A.\,Joyal~\cite{Joyal1234} and
R.\,Hasegawa~\cite{Hasegawa} for analytic endofunctors on the category of
sets, we give two characterisations of analytic functors between presheaf
categories over groupoids: 
(i)~as functors preserving filtered colimits, quasi-pullbacks, and cofiltered
limits; 
and
(ii)~as functors preserving filtered colimits and wide quasi-pullbacks.  
The development establishes that small groupoids, analytic functors between
their presheaf categories, and quasi-cartesian natural transformations between
them form a 
\mbox{$2$-category}.
\end{abstract}


\maketitle

\section{Introduction}

The concept of multivariate analytic functor on the category $\Set$ of
sets and functions was introduced by A.\,Joyal in~\cite{Joyal1234} to
provide a conceptual basis for his theory of combinatorial species of
structures~\cite{JoyalAdv,BLL}.  

A species of structures is a functor from the category 
of finite sets and bijections to $\Set$.  
These can be equivalently presented as functors from the category 
of finite cardinals and permutations to $\Set$, or as symmetric sequences 
$$
P 
= 
\bigsetof{ \ 
  P_n\times \mathfrak S_n \rightarrow P_n 
  : \enspace (p,\sigma) \,\mapsto\, p\act_P\sigma
  \ }_{n\in\Nat}
$$
given by families of set-theoretic representations of the symmetric groups.
Here, the sets $P_n$ are thought of as a species of combinatorial structures
$P$ on an $n$-element set, while the symmetric-group representations induce
isomorphism types that correspond to their unlabelled version.
In general, for a species $P$ and a set of labels $X$, the set of
$X$-labelled $P$-structures is given by 
\begin{equation}\label{LabelledStructures}
\begin{array}{rcll}
\lan P\, X 
&\eqdef&
\sum_{n\in\Nat}\,
  P_n
  \!\mathrel{\times\hspace{-2.5mm}_{\raisebox{-1.35mm}{\tiny$\mathfrak S_n$}}}\!
  X^n
&\qquad(X\in\Set)
\end{array}
\end{equation}
where 
$P_n
\!\mathrel{\times\hspace{-2.5mm}_{\raisebox{-1.35mm}{\tiny$\mathfrak S_n$}}}\!
X^n$
denotes the quotient of $P_n\times X^n$ by the
equivalence relation identifying 
$\big(p,(x_{\sigma 1},\ldots,x_{\sigma n})\big)$ with 
$\big(p\act_P\sigma,(x_1,\ldots,x_n)\big)$ for all $\sigma\in\mathfrak
S_n$, $p\in P_n$, and $x_1,\ldots,x_n\in X$.  In particular, the set $\lan
P\,1$ for a singleton set~$1$ corresponds to that of unlabelled
$P$-structures.

An endofunctor on $\Set$ is said to be analytic if it has a Taylor series
development as in~(\ref{LabelledStructures}) above; that is, if it is
naturally isomorphic to $\lan P$ for some species $P$.  One 
respectively regards species of structures and analytic functors as
combinatorial versions of formal exponential power series and exponential
generating functions.  
%
A.\,Joyal characterised the analytic endofunctors on $\Set$ as those that
preserve filtered colimits, cofiltered limits, and quasi-pullbacks
(equivalently, weak pullbacks).

In~\cite{JoyalAdv}, A.\,Joyal also introduced the notion of a linear
species as a functor from the category of finite linear orders and
monotone bijections to $\Set$; equivalently, an $\Nat$-indexed
family of sets.  Every linear species $L$ 
freely induces a species $L\!\times\!\mathfrak S$ as follows
$$\begin{array}{ll}
(L_n\times\mathfrak S_n)\times\mathfrak S_n
\rightarrow
(L_n\times\mathfrak S_n)
:\ 
\big((\ell,\sigma),\sigma'\big)
\mapsto
\big(\ell,\sigma\act\sigma'\big)
& \qquad(n\in\Nat)
\end{array}$$
Its associated analytic endofunctor $\lan{L\!\times\!\mathfrak S}$ on
$\Set$ is of the form
\begin{equation}\label{GeneratingFunctor}
\begin{array}{rcll}
\lan{L\!\times\!\mathfrak S}\,(X)
&\iso&
\sum_{n\in\Nat}\,
  L_n\times X^n
&\qquad(X\in\Set)
\end{array}
\end{equation}
Thus, one respectively regards linear species and their induced analytic
functors as combinatorial versions of formal power series and generating
functions.  

Independently of the above considerations, the multivariate version of
functors on $\Set$ of the form~(\ref{GeneratingFunctor}) was introduced by
J.-Y.\,Girard in~\cite{Girard} also under the name of analytic functors.
These he characterised as those that preserve filtered colimits, wide
pullbacks, and equalisers.  In~\cite{Taylor}, P.\,Taylor tighten this
characterisation remarking that the preservation of equalisers was redundant. 
R.\,Hasegawa revisited the characterisation of Joyal's analytic
endofunctors on $\Set$ in this light in~\cite{Hasegawa}, observing that
they can be also characterised as those preserving filtered colimits and
weak wide pullbacks (equivalently, wide quasi-pullbacks).
The development of J.-Y.\,Girard 
put this line of work in the context of categorical stable domain theory
(as so did explicitly the subsequent work of P.\,Taylor) 
and was a preliminary step leading to linear logic~\cite{GirardLL}.

A bicategorical framework for the above body of work was put forward by
G.\,L.\,Cattani and G.\,Winskel in~\cite{CattaniWinskel} from the
perspective of presheaf models for concurrency and by M.\,Fiore, N.\,Gambino,
M.\,Hyland and G.\,Winskel in~\cite{Esp} from the viewpoint of species of
structures.  The work reported here supplements the latter one.  Indeed,
we generalise the aforementioned characterisations of analytic
endofunctors on $\Set$ to analytic functors between presheaf categories
over groupoids~(Theorem~\ref{CharacterisationTheorem}); and, 
in this context, exhibit an equivalence of categories between generalised
species of structures and natural transformations, and analytic functors
and quasi-cartesian natural
transformations~(Corollary~\ref{EquivalenceCorollary}).  
This leads to the \mbox{$2$-category} of small groupoids, analytic
functors between their presheaf categories, and quasi-cartesian natural
transformations between them (Corollary~\ref{AFCorollary}), placing the
subject in the context of categorical stable domain theory and providing
\mbox{$2$-dimensional} models of a rich variety of computational
structures (Remark~\ref{FinalRemark}).

The paper contributes thus to one of the many fundamental structures
researched by Glynn Winskel in his work on the mathematical understanding
and modelling of processes.

\section{Free symmetric strict monoidal completion}

We let $\freesmc{}$ be the left adjoint to the forgetful functor from the
category of symmetric strict monoidal small categories and strong monoidal
functors to the category~$\Cat$ of small categories and functors.  For a
small category $\scatC$, the unit of this adjunction is denoted
$\listof{\fstarg}:\scatC\rightarrow\freesmc\scatC$.

The category $\freesmc\scatC$ can be explicitly described by the Grothendieck
construction~\cite{Grothendieck} applied to the functor 
$\scatC^{(\_)}:\Bij\rightarrow\Cat: n \mapsto \scatC^n$ for $\Bij$ the
category of finite cardinals and permutations. 
That is, $\freesmc\scatC$ has objects given by functions 
$C:\card C\rightarrow\scatC$ with $\card C$ in $\Bij$ and morphisms
$\gamma = (\indmap\gamma,\objmap\gamma):C\rightarrow C'$
given as in the following diagram 
$$\xymatrix{
\card C \ar[rr]^-{\indmap\gamma} \ar[rd]_-C &
\ar@{}[d]|(.35){\stackrel{\objmap\gamma}{\natrightarrow}} & 
\card{C'} \ar[dl]^-{C'} \\
& \scatC & 
}$$
with $\indmap\gamma$ in $\Bij$.
Identities are given by 
the maps $(\id_{\card C},\id_C)$, 
while diagrammatic composition is given by 
$\alpha\dcomp\beta 
 \eqdef\big(\indmap\alpha\dcomp\indmap\beta,
      \objmap\alpha\dcomp\objmap\beta_{\indmap\alpha}\big)$.
Thus, 
maps and their composition 
can be visualised as follows
$$
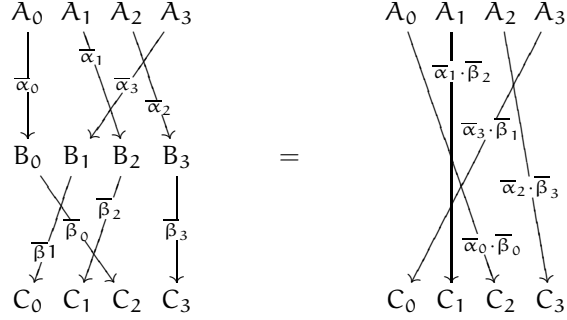
\xymatrix@C=0pt@R=40pt{
A_0 \ar[d]|-{\objmap\alpha_0} & A_1 \ar[dr]|(.3){\objmap\alpha_1} & 
A_2 \ar[dr]|(.65){\objmap\alpha_2} & A_3 \ar[dll]|-{\objmap\alpha_3} 
&&&&
A_0 \ar[ddrr]|(.8){\hspace{3mm}\objmap\alpha_0\dcomp\objmap\beta_0} & 
A_1 \ar[dd]|(.2){\hspace{3mm}\objmap\alpha_1\dcomp\objmap\beta_2} & 
A_2 \ar[ddr]|(.6){\objmap\alpha_2\dcomp\objmap\beta_3} 
& A_3 \ar[ddlll]|(.4){\objmap\alpha_3\dcomp\objmap\beta_1} 
\\
B_0 \ar[drr]|(.5){\objmap\beta_0} & B_1 \ar[dl]|(.65){\objmap\beta1} 
& B_2 \ar[dl]|(.35){\objmap\beta_2} & B_3 \ar[d]|-{\objmap\beta_3} 
&\qquad& = &\qquad& 
\\
C_0 & C_1 & C_2 & C_3 
&&&&
C_0 & C_1 & C_2 & C_3 \\
}
$$

The strict symmetric monoidal structure of $\freesmc\scatC$
has as unit object the empty function $0\rightarrow\scatC$, as tensor
product~$\oplus$ the construction 
$\copair{C,C'}: \card C+\card{C'}\rightarrow\scatC$, and as symmetry
the maps 
$$\xymatrix{
\card C + \card{C'} \ar[rd]_-{\copair{C,C'}}
\ar[rr]^-{\copair{\inj 2,\inj 1}} & 
\ar@{}[d]|(.35){\stackrel{\id}{\Rightarrow}} & 
\card{C'}+\card C \ar[dl]^-{\copair{C',C}}
\\
& \scatC & 
}$$
where $+$ denotes the sum of cardinals, with injections $\inj1$ and $\inj2$,
and copairing $\copair{\fstarg,\sndarg}$.

\medskip
We write $\PSh\scatC$ for the \Rem{presheaf category} $\Set^{\scatC^\op}$ over
a small category $\scatC$.
By the universal property of $\freesmc\scatC$, the Yoneda embedding 
$\yon_\scatC: \scatC \rightembedding \PSh\scatC$ extends as a 
(strong symmetric monoidal) \Rem{sum functor}
$\Sf_\scatC:\freesmc\scatC\rightarrow\PSh\scatC$ (with respect to the
coproduct symmetric monoidal structure of $\PSh\scatC$) as follows
$$\xymatrix{
\scatC \ar@{^(->}[rd]_-{\yon_\scatC} \ar[r]^-{\scriptlistof{\fstarg}}
\ar[r]|-{\hspace{6mm}\raisebox{-12mm}{\scriptsize$\iso$}} & \freesmc\scatC
\ar[d]^-{\Sf_\scatC} \\
& \PSh\scatC
}$$
where
$$\begin{array}{rcll}
\Sf_\scatC(C) & \eqdef & \sum_{i\in\card C}\ \yon_\scatC(C_i)
& \qquad(C\in\freesmc\scatC)
\end{array}$$

Examining the sum functor, one notes that, for $A, B \in \freesmc\scatC$, 
\begin{equation}\label{PShC[SA,SB]}
\begin{array}{rclrcl}
\PSh\scatC\hom{\Sf A,\Sf B}
& \!\!\iso\!\! & 
  \prod_{i\in\card A} \PSh\scatC\hom{\yon(A_i),\Sf B}
& \!\!\iso\!\! & 
  \prod_{i\in\card A} \Sf B(A_i)
\\[2mm]
& \!\!\iso\!\! & 
  \prod_{i\in\card A} \sum_{j\in\card B} \scatC\hom{A_i,B_j}
& \!\!\iso\!\! & 
  \sum_{\varphi\in{\card B}^{\card A}}
    \prod_{i\in\card A}\scatC\hom{A_i,B_{\varphi i}}
\end{array}
\end{equation}
In other words, the full subcategory of $\PSh\scatC$ determined by the set of
objects 
$\setof{\, \Sf C \in \PSh\scatC \suchthat C \in \freesmc\scatC \,}$ is
the free finite coproduct completion of $\scatC$. 

By means of the projection map
$$\textstyle
\sum_{\varphi\in{\card B}^{\card A}}
  \prod_{i\in\card A}\scatC\hom{A_i,B_{\varphi i}}
\ \rightarrow \
\Set(\,\card A,\card B\,)
\, :\, \big(\varphi,\pair{f_i}_{i\in\card A}\big) \,\mapsto\, \varphi
$$
the isomorphism~(\ref{PShC[SA,SB]}) induces a map 
$$
\PSh\scatC\hom{\Sf A,\Sf B}\rightarrow\,\Set(\,\card A,\card B\,)
$$
that associates an \emph{underlying function}~$\card A\rightarrow\card B$ 
to every morphism $\Sf A\rightarrow\Sf B$ in $\PSh\scatC$.

\begin{definition}
For $A,B\in\freesmc\scatC$, we say that $\Sf A \rightarrow \Sf B$ in
$\PSh\scatC$ is \Rem{injective, surjective, or bijective on indices} whenever
its underlying function $\card A \rightarrow \card B$ 
is.  
\end{definition}

\begin{proposition}\label{S_Faithful_And_Conservative}
\begin{enumerate}[(i)]
\item\label{S_Faithful}
The sum functor is faithful.

\item\label{S_Conservative}
If $f: \Sf A\rightarrow \Sf B$ in $\PSh\scatC$ is bijective on indices
then there exists a (necessarily unique) $\gamma: A \rightarrow B$ in 
$\freesmc\scatC$ such that $\Sf\gamma = f$.  Hence, the sum functor is
conservative.
\end{enumerate}
\end{proposition}

\begin{proposition}\label{BijSurInjOnIndices}
\begin{enumerate}[(i)]
  \item \label{BijSurOnIndices}
    For a small category $\scatA$ and $A, A'\in \freesmc\scatA$, every epi
    (resp.~iso) $\Sf A\rightarrow \Sf A'$ in $\PSh\scatA$ is surjective
    (resp.~bijective) on indices.

  \item \label{InjOnIndices}
    For a small groupoid $\gpdG$ and $G, G'\in \freesmc\gpdG$, every mono
    $\Sf G\rightarrow\Sf G'$ in $\PSh\gpdG$ is injective on indices.
\end{enumerate}
\end{proposition}

\section{Analytic 
  functors}

We recall the notion of analytic functor between presheaf categories
introduced in~\cite{Esp}. 
These analytic functors generalise the ones previously introduced by A.\,Joyal
between categories of indexed sets and sets~\cite[\S~1.1]{Joyal1234}, and
are the central structure of study in the paper.   
\begin{definition}
A functor 
$\PSh\scatA\rightarrow\PSh\scatB$ is said to be \Def{analytic} if it
appears in a left Kan extension as follows  
$$\xymatrix{
\freesmc\scatA \ar[r]^-{\Sf_\scatA} \ar[rd] 
\ar@{}[r]|-{\hspace{4.5mm}\raisebox{-13mm}{$\stackrel{\mbox{\tiny$\mathrm{Lan}$}}{\Rightarrow}$}}
& \PSh\scatA \ar[d] \\
& \PSh\scatB
}$$
for some 
functor $\freesmc\scatA\rightarrow\PSh\scatB$.
\end{definition}

That is, analytic functors between presheaf categories are those
naturally isomorphic to the functors $\lan P:\PSh\scatA\rightarrow\PSh\scatB$
given by the following coend 
\begin{equation}\label{P!coend}\begin{array}{rcll}
\lan P\,X\,b & \eqdef &  
\coend^{A\in\freesmc\scatA} P\,A\,b \times \PSh\scatA\bighom{\Sf_\scatA(A),X}
& \qquad\qquad(X\in\PSh\scatA, b\in\scatB^\op)
\end{array}\end{equation}
for some $P: \freesmc\scatA\rightarrow\PSh\scatB$.  

\begin{notation}
For a functor $F: \lscat C\rightarrow\PSh\scatC$ it will be convenient to use
the following notational conventions.  For morphisms $f:A\rightarrow B$ in
$\lscat C$ and $g: c\rightarrow d$ in $\scatC$, and for an element $x\in
F\,A\,d$, we set
$x\act_F f \eqdef (F\,f)_d(x)\in F\,B\,d$; 
$g\act_F x \eqdef F\,A\,g\,(x)\in F\,A\,c$; 
and 
$g\act_F x\act_F f 
   \eqdef 
     g\act_F(x\act_F f) = (g\act_F x)\act_F f \in F\,B\,c$.
\end{notation}

Henceforth, we will use the following explicit description of the 
coend~(\ref{P!coend}):
$$\textstyle
\big(
  \sum_{A\in\freesmc\scatA}
    P\,A\,b \times \PSh\scatA\bighom{\Sf_\scatA(A),X}
\hspace{.3mm}\big)_{\mbox{\large/}\,\approx}
\qquad\qquad(X\in\PSh\scatA, b\in\scatB^\op)
$$
where $\approx$ is the equivalence relation generated by 
\begin{equation}\label{Coend_Equivalence}\begin{array}{rcl}
(A,p,\Sf_\scatA(\alpha)\dcomp x)
& \sim &
(A',p\act_P\alpha,x)
\end{array}\end{equation}
for all $\alpha:A\rightarrow A'$ in $\freesmc\scatA$, $p\in P\,A\,b$,
and $x:\Sf_\scatA(A')\rightarrow X$ in $\PSh\scatA$.
Further, we write $p\ntensor{A}x$ for the equivalence class of $(A,p,x)$.
Under this convention, the identification~(\ref{Coend_Equivalence}) amounts to
the identity
$$\begin{array}{rcl}
p\ntensor{A}\big(\Sf_\scatA(\alpha)\dcomp x\big)
& = & 
\big(p\act_P\alpha\big)\ntensor{{A\!'}}x
\end{array}$$
and the functorial action of $\lan P$ is given by 
$$\begin{array}{rcl}
\beta \act_{\lan P} (p\ntensor{A}x) \act_{\lan P} f
& \eqdef & 
(\beta \act_P p) \ntensor{A} (x \dcomp f)
\end{array}$$
for all $(p\ntensor{A}x)\in\lan P\,X\,b$, $f:X\rightarrow X'$ in
$\PSh\scatA$ and $\beta:b'\rightarrow b$ in $\scatB$.

\begin{notation}
For categories $\lscat A$ and $\lscat B$, we let $\CAT\hom{\lscat A,\lscat B}$
denote the category of functors $\lscat A\rightarrow\lscat B$ and natural
transformations between them.
\end{notation}

\begin{proposition}
The functor 
$\lan{(\fstarg)}:
\CAT\bighom{\freesmc\scatA,\PSh\scatB}\rightarrow\CAT\bighom{\PSh\scatA,\PSh\scatB}$
is faithful.
\end{proposition}
This is a consequence of the following.
\begin{lemma}\label{Invariance_Lemma}
Let $P: \freesmc\scatA\rightarrow\PSh\scatB$.
For $\alpha_0: A_0\rightarrow A$ in $\freesmc\scatA$, and  
$p_0\ntensor{A_0}\Sf(\alpha_0)$ and $p_1\ntensor{A_1}x_1$ in 
$\lan P(\Sf A)(b)$, if 
$p_0\ntensor{A_0}\Sf(\alpha_0) = p_1\ntensor{A_1}x_1$ then there exists
(a necessarily unique) $\alpha_1: A_1 \rightarrow A$ in
$\freesmc\scatA$ such that $x_1 = \Sf(\alpha_1)$ and 
$p_0\act_P \alpha_0 = p_1 \act_P\alpha_1$.  
\end{lemma}
\begin{proof}
It is enough to establish the lemma in the following two cases.
\begin{myitemize}
\item
When there exists $\alpha: A_1\rightarrow A_0$ in $\freesmc\scatA$ such that
$p_1\act_P\alpha = p_0$ and $\Sf(\alpha) \dcomp \Sf(\alpha_0) = x_1$. 
In which case, taking $\alpha_1=\alpha\dcomp\alpha_0$ we are done. 

\item
When there exists $\alpha: A_0\rightarrow A_1$ in $\freesmc\scatA$ such that
$p_0\act_P\alpha = p_1$ and $\Sf\alpha\dcomp x_1 = S(\alpha_0)$.
In which case, $x_1: \Sf(A_1)\rightarrow \Sf(A)$ in $\PSh\scatA$ is bijective on
indices and hence, by
Proposition~\ref{S_Faithful_And_Conservative}(\ref{S_Conservative}), there
exists $\alpha_1: A_1 \rightarrow A$ in $\freesmc\scatA$ such that
$\Sf(\alpha_1) = x_1$.  Further, by
Proposition~\ref{S_Faithful_And_Conservative}(\ref{S_Faithful}), we have that
$\alpha_0 = \alpha\dcomp\alpha_1$ and hence also that
$p_0\act_P\alpha_0 
   = p\act_P(\alpha\dcomp\alpha_1) 
   = (p\act_P\alpha)\act_P\alpha
   = p_1\act_P\alpha_1$. \qed
\end{myitemize}
\end{proof}
\begin{corollary}\label{p*id=p'*id_=>_p=p'}
For $P:\freesmc\scatA\rightarrow\PSh\scatB$, if 
$p\ntensor{A}\id_{\Sf A} = p'\ntensor{A}\id_{\Sf A}$ in 
$\lan P(\Sf A)(b)$ then $p = p'$ in $P\,A\,b$.
\end{corollary}

\section{Coefficients functors}

Via the canonical natural isomorphisms
$$\begin{array}{rclcll}
\PSh\scatA\bighom{\Sf A,X}
& \!\iso\! & \textstyle
\prod_{i\in\card A}\PSh\scatA\bighom{\yon(A_i),X}
& \!\iso\! & \textstyle
\prod_{i\in\card A} X(A_i)
& \qquad(A\in\freesmc\scatA, X\in\PSh\scatA)
\end{array}$$
every analytic functor $F:\PSh\scatA\rightarrow\PSh\scatB$ admits a
\Rem{Taylor series development} as follows
\begin{equation}\label{TaylorSeriesDevelopment}
\begin{array}{rcll}
F\,X\,b 
& \iso & 
\Big(\sum_{n\in\Nat}\sum_{a_1,\ldots,a_n\in\scatA}
  P\big(\!\oplus_{i=1}^n\listof{a_i}\big)(b)
  \times 
  \prod_{i=1}^n X(a_i)
  \Big)_{\!\mbox{\large/}\approx}
  & 
\qquad
(X\in\PSh\scatA, b\in\scatB^\op)
\end{array}
\end{equation}
for some \Rem{coefficients functor} $P:\freesmc\scatA\rightarrow\PSh\scatB$
(referred to as an $(\scatA,\scatB)$-\Rem{species of structures}
in~\cite{FioreFOSSACS,Esp}).
%
The representation of analytic functors~(\ref{TaylorSeriesDevelopment}) for
$\scatA$ a finite discrete category and $\scatB$ the one-object category
directly exhibits them as the multivariate analytic functors
of A.\,Joyal~\cite[\S~1.1]{Joyal1234}.

\medskip
The 
coefficients functors of an analytic functor are unique up to isomorphism.   
\begin{proposition}\label{Uniqueness_Of_Coefficients}
The functor 
$\lan{(\fstarg)}:
\CAT\bighom{\freesmc\scatA,\PSh\scatB}\rightarrow\CAT\bighom{\PSh\scatA,\PSh\scatB}$
is conservative.  
That is, for $P, Q: \freesmc\scatA\rightarrow\PSh\scatB$, if  
$\lan P \iso \lan Q:\PSh\scatA\rightarrow\PSh\scatB$ then 
$P\iso Q$.
\end{proposition}
This result is a corollary of
Proposition~\ref{Uniqueness_Of_Coefficients_Generalisation} below, for which
we need to recall that a natural transformation is said to be
\Def{quasi-cartesian} whenever all its naturality squares are quasi-pullbacks,
where a \Def{quasi-pullback} is a commutative square for which the unique
mediating morphism from its span to the pullback of its cospan is an
epimorphism. 

The notion of quasi-pullback in presheaf categories is given pointwise.
\begin{lemma}
For a small category $\scatC$, a commutative square in $\PSh\scatC$ as on the
left below
\begin{center}
\mbox{}\hfill\hfill
$\xymatrix{
  \ar[d]_-h Q\ar[r]^-k & Y \ar[d]^-g
  \\
  X \ar[r]_f & Z
}$
\hfill\hfill
$\xymatrix{
  \ar[d]_-{h_c} Q\,c\ar[r]^-{k_c} & Y\,c \ar[d]^-{g_c}
  \\
  X\,c \ar[r]_{f_c} & Z\,c
}$
\hfill\hfill\mbox{}
\end{center}
is a quasi-pullback iff so are the commutative squares in $\Set$ as on the
right above for every $c\in\scatC$.
\end{lemma}
\begin{proof}
Follows from the facts that in presheaf categories limits and colimits are
given pointwise and that the functors that evaluate presheaves at an object
preserve them. \qed
\end{proof}

\begin{proposition}\label{Uniqueness_Of_Coefficients_Generalisation}
Let $P, Q: \freesmc\scatA\rightarrow\PSh\scatB$ 
and 
$\varphi:\lan P\natrightarrow\lan Q:\PSh\scatA\rightarrow\PSh\scatB$. 
For the following statements:
\begin{enumerate}[(i)]
\item \label{qc_hyp}
  The natural transformation $\varphi$ is quasi-cartesian.

\item \label{simple_qc_char}
  For every $A\in\freesmc\scatA$, $b\in\scatB^\op$, and $p\in P\,A\,b$ there
  exists (a necessarily unique) $q\in Q\,A\,b$ such that
  $\varphi\big(p\ntensor{A}\id_{\Sf A}\big) 
     = q\ntensor{A}\id_{\Sf A}$. 

\item \label{qc_lift}
  There exists a (necessarily unique) natural transformation
  $\phi:P\natrightarrow Q:\freesmc\scatA\rightarrow\PSh\scatB$ such
  that $\varphi = \lan\phi$. 
\end{enumerate}
we have 
that~$\emph{(\ref{qc_hyp})}\Rightarrow\emph{(\ref{simple_qc_char})}\Rightarrow\emph{(\ref{qc_lift})}$.
\end{proposition}
\begin{proof}
$\mbox{(\ref{qc_hyp})}\Rightarrow\mbox{(\ref{simple_qc_char})}$
For $p\in P\,A\,b$, let 
$$\begin{array}{rcl}
\varphi_{\Sf A,b}(p\ntensor{A}\id_{\Sf A}) 
& = & 
(q\ntensor{A'}\id_{\Sf A'}) \act_{\lan P} s
\end{array}$$
for $q\in Q(A')(b)$ and $s:\Sf A'\rightarrow\Sf A$ in $\PSh\scatA$.

Since $\varphi$ is quasi-cartesian, there exists
$(p_0\ntensor{A_0}s_0) \in \lan P(\Sf A')(b)$
such that 
\begin{eqnarray}\label{p0*(s0.s)=p*id}
p_0\ntensor{A_0}(s_0\dcomp s) & = & p\ntensor{A}\id_{\Sf A}
\end{eqnarray}
and 
$$\begin{array}{rcl}
\varphi_{\Sf A',b}(p_0\ntensor{A_0}s_0) & = & q\ntensor{A'}\id_{\Sf A'}
\end{array}$$

From~(\ref{p0*(s0.s)=p*id}), by Lemma~\ref{Invariance_Lemma}, there
exists $\alpha_0: A_0\rightarrow A$ in $\freesmc\scatA$ such that
\begin{eqnarray}\label{s0.s=S(alpha)}
s_0\dcomp s & = & \Sf(\alpha_0)
\end{eqnarray}
and $p_0\act_P \alpha_0 = p$.  In particular, thus, 
$s_0: \Sf(A_0)\rightarrow \Sf(A')$ in $\PSh\scatA$ is injective on
indices. 

Now, let 
$$\begin{array}{rcl}
\varphi_{\Sf(A_0),b}(p_0\ntensor{A_0}\id_{\Sf(A_0)})
& = & 
q_1\ntensor{A_1}s_1
\end{array}$$
for $q_1\in Q(A_1)(b)$ and $s_1:\Sf(A_1)\rightarrow\Sf(A_0)$ in
$\PSh\scatA$.  By naturality of $\varphi$, we have that 
$$\begin{array}{rcl}
q_1\ntensor{A_1}(s_1\dcomp s_0) & = & q\ntensor{A'}\id_{\Sf(A')}
\end{array}$$
and, by Lemma~\ref{Invariance_Lemma}, that there exists
$\alpha_1: A_1\rightarrow A'$ in $\freesmc\scatA$ such that 
$s_1\dcomp s_0 = \Sf(\alpha_1)$ and $q_1\act_Q\alpha_1 = q$.  In
particular, thus, $s_0$ is surjective, and hence bijective, on
indices.  It then follows from~(\ref{s0.s=S(alpha)}) that also $s$ is
bijective on indices and hence that there exists 
$\alpha: A' \rightarrow A$ such that $\Sf\alpha = s$.

Thus, 
$\varphi_{\Sf A,b}(p\ntensor{A}\id_{\Sf A}) 
 =  
(q\act_Q \alpha)\ntensor{A}\id_{\Sf A}$.

\medskip
$\mbox{(\ref{simple_qc_char})}\Rightarrow\mbox{(\ref{qc_lift})}$
The family of mappings 
$\phi_{A,b}: P\,A\,b \rightarrow Q\,A\,b$~($A\in\freesmc\scatA,
b\in\scatB^\op$) 
associating $p\in P(A)(b)$ with the unique $q\in Q(A)(b)$ such that 
$\varphi_{\Sf A,b}\big(p\ntensor{A}\id_{SA}\big) = (q\ntensor{A}\id_{\Sf A})$
determine a natural transformation $\phi: P \natrightarrow Q$ with the
desired property. \qed  
\end{proof}

\smallskip
It is interesting to note that not every natural transformation 
in the image of 
${\lan{(\fstarg)}:
   \CAT\bighom{\freesmc\scatA,\PSh\scatB}
   \rightarrow
   \CAT\bighom{\PSh\scatA,\PSh\scatB}}$
is quasi-cartesian.  
Indeed, for 
${\Scat \eqdef \big(\bot\rightarrow\top\big)}$,
$P\eqdef\freesmc\Scat\bighom{\listof{\top},\fstarg}$,
and 
$\phi: P\natrightarrow 1: \freesmc\Scat\rightarrow\Set$, 
the naturality square associated to
$\lan\phi: \lan P \natrightarrow \lan 1: \PSh\Scat\rightarrow\Set$
induced by $\yon(\bot)\rightarrow \yon(\top)$ in
$\PSh\Scat$ 
is not a quasi-pullback.  However, we have the following result.
\begin{proposition}
For $\phi: P \natrightarrow Q:\freesmc\gpdG\rightarrow\PSh\scatC$ where
$\gpdG$ is a small groupoid, the natural transformation 
$\lan\phi: \lan P \natrightarrow \lan Q:\PSh\gpdG\rightarrow\PSh\scatC$
is quasi-cartesian.
\end{proposition}
\begin{proof}
For $f: X \rightarrow Y$ in $\PSh\gpdG$, let 
$(p\ntensor{G}y)\in \lan P\,Y\,b$ and $(q\ntensor{G'}x)\in \lan Q\,X\,b$ be such
that 
$$\begin{array}{rclclcl}
{\phi_{G,b}(p)\ntensor{G}y} 
\ = \ 
\lan\phi_{Y,b}(p\ntensor{G}y) 
\ = \ 
(q\ntensor{G'}x)\act_{\lan Q}f
\ = \ 
q\ntensor{G'}(x\dcomp f)
\end{array}$$
Then, as $\gpdG$ is a groupoid, it follows that there exists 
$\sigma: G \rightarrow G'$ in $\freesmc\gpdG$ such that
$\phi_{G,b}(p)\act_Q\sigma = q$ and $y = \Sf(\sigma)\dcomp x\dcomp f$.

Since, for 
$p \ntensor{G} (\Sf\sigma\dcomp x) = (p\act_P\sigma)\ntensor{G'}x$ in
$\lan P\,X\,b$ we have that 
$(p \ntensor{G} (\Sf\sigma\dcomp x))\act_{\lan P}f 
   = {p \ntensor{G} (\Sf\sigma\dcomp x\dcomp f)}
   = p \ntensor{G} y$
and 
$\lan\phi_{X,b}\big((p\act_P\sigma)\ntensor{G'}x\big)
   = \big(\phi_{G',b}(p\act_P\sigma)\big)\ntensor{G'}x
   = \big(\phi_{G,b}(p)\act_P\sigma\big)\ntensor{G'}x
   = q\ntensor{G'}x$
we are done. \qed
\end{proof}

Quasi-cartesian natural transformations are closed under vertical composition
and we are naturally led to introduce the following.
\begin{definition}
For small categories $\scatA$ and $\scat B$, we let
$\AF\hom{\scatA,\scatB}$ be the subcategory of
$\CAT\bighom{\PSh\scatA,\PSh\scatB}$ consisting of analytic functors and
quasi-cartesian natural transformations between them.  
\end{definition}

\begin{corollary}\label{lan_Full_And_Faithful}
For $\gpdG$ a small groupoid, 
the functor
$\lan{(\fstarg)}:
   \CAT\bighom{\freesmc\gpdG,\PSh\scatC}
   \rightarrow
   \CAT\bighom{\PSh\gpdG,\PSh\scatC}$
restricts to an essentially surjective, full and faithful functor
\begin{equation}\label{EssSurFullFaithLan}
\lan{(\fstarg)}:
   \CAT\bighom{\freesmc\gpdG,\PSh\scatC}
   \rightarrow
   \AF\bighom{\gpdG,\scatC}
\end{equation}
\end{corollary}

\section{Generic coefficients functor}

We proceed to construct a quasi-inverse to~(\ref{EssSurFullFaithLan}) when the
small category $\scatC$ 
is a groupoid.
The central notion isolated by A.\,Joyal for this purpose is that of generic
element~\cite[Appendice, D\'efinition~2]{Joyal1234}.
\begin{definition}
For $F:\PSh\scatA\rightarrow\PSh\scatB$, we say that $x\in F\,X\,b$ is
\Def{generic} if for every cospan $f: X\rightarrow Z \leftarrow Y: g$ in
$\PSh\scatA$ and $y\in F\,Y\,b$ such that $x\act_F f = y\act_F g$ there exists
$h: X\rightarrow Y$ in $\PSh\scatA$ such that $f = h\dcomp g$ and $x\act_F h =
y$. 
\end{definition}
For instance, it follows from the proposition below that for
$P:\freesmc\gpdG\rightarrow\PSh{\scatC}$ with $\scat G$ a small groupoid,
$G\in\freesmc\gpdG$, and $c\in\scatC$, the generic elements in 
$\lan P(\Sf G)(c)$ are of the form $p\ntensor G\id_{\Sf G}$ for $p\in
P\,G\,c$.

\begin{proposition}
For $P:\freesmc\gpdG\rightarrow\PSh\scatC$ with $\gpdG$ a small groupoid, 
$(p\ntensor{G}x)\in\lan P\,X\,c$ is generic iff 
${x:\Sf G\rightarrow X}$ in $\PSh\gpdG$ is an isomorphism.
\end{proposition}
\begin{proof}
($\Rightarrow$) 
Let $(p\ntensor Gx)\in \lan P\,X\,c$ be generic.  As $(p\ntensor G x) =
(p\ntensor G \id_{\Sf G})\act_{\lan P} x$, there exists $h:X\rightarrow \Sf G$
such that $h\act x=\id_X$ and 
$p\ntensor G(x\act h) 
   = (p\ntensor G x)\act_{\lan P} h 
   = (p\ntensor G \id_{\Sf G})$.
The latter identity implies that $x\act h$ is an automorphism on $\Sf G$, and
we are done.

($\Leftarrow$)
Let $(p\ntensor{G}x)\in\lan P\,X\,c$ with $x$ an isomorphism.  Consider
a cospan $f: X\rightarrow Z \leftarrow Y: g$ and 
$(q\ntensor H y)\in \lan P\,Y\,c$ with 
$\big(p\ntensor G(x\act f)\big) 
   = \big((p\ntensor G x)\act_{\lan P} f\big) 
   = \big((q\ntensor H y)\act_{\lan P} g\big) 
   = \big(q\ntensor H(y\act g)\big)$.  
Then, there exists $\sigma:G\rightarrow H$ in $\freesmc\gpdG$ such that
$p\act_P\sigma = q$ and $x\act f = (\Sf\sigma)\act y\act g$; and the map
$x^{-1}\act(\Sf\sigma)\act
y:X\rightarrow Y$ has the desired properties. \qed
\end{proof}

\begin{lemma}\label{Generic=>Minimal}
Let $F:\PSh\scatA\rightarrow\PSh\scatB$.  For every $x\in F\,X\,b$ generic,
$y\in F\,Y\,b$, and $f: Y\rightarrow X$ in $\PSh\scatA$ such that $y\act_F f =
x$, one has that $f$ is split epi. 
\end{lemma}
\begin{proof}
Because the hypotheses imply the existence of $h: X\rightarrow Y$ 
such that $x\act_F h = y$ and $h \dcomp f = \id_X$. \qed
\end{proof}

We now explain how analytic functors from presheaf categories over groupoids
are engendered by their compact generic elements uniquely up to isomorphism. 
\begin{definition}
A functor $F:\PSh\scatA\rightarrow\PSh\scatB$ is said to be \noRem{engendered
by its (compact) generic elements} whenever for every $x\in F\,X\,b$ there
exists a generic element $x_0\in F(X_0)(b)$ (with $X_0=\Sf A$ for
$A\in\freesmc\scatA$) and $f:X_0\rightarrow X$ in $\PSh\scatA$ such that
$x_0\act_F f=x$.
\end{definition}
\begin{proposition}
Let $F:\PSh\scatA\rightarrow\PSh\scatB$.  For $x\in F(\Sf A)(b)$ and
$x'\in F(\Sf A')(b)$ both generic, and for $f:\Sf A\rightarrow X$ and
$f':\Sf A'\rightarrow X$ in $\PSh\scatA$ such that $x\act_F f = x' \act_F f'$, 
there exists a split epi $\alpha: A\rightarrow A'$ in
$\freesmc\scatA$ such that $x\act_F\Sf(\alpha) = x'$ and 
$f = \Sf(\alpha)\dcomp f'$. 
\end{proposition}
\begin{proof}
Since $x$ is generic, there exists $g:\Sf(A)\rightarrow \Sf(A')$ in
$\PSh\scatA$ such that $x\act_F g = x'$ and $g\dcomp f'=f$.  Further,
since $x'$ is generic, by Lemma~\ref{Generic=>Minimal}, $g$ is
split epi.  Analogously, since $x'$ is generic, there exists 
$g':\Sf(A')\rightarrow\Sf(A)$ in $\PSh\scatA$ such that 
$x'\act_F g' = x$ and $g'\dcomp f = f'$.  Further, since $x$ is
generic, by Lemma~\ref{Generic=>Minimal}, $g'$ is split epi.

As $g:\Sf(A)\rightarrow\Sf(A')$ and $g':\Sf(A')\rightarrow\Sf(A)$ in
$\PSh\scatA$ are both surjective, and hence bijective, on indices, 
there exist $\alpha:A\rightarrow A'$ and $\alpha':A'\rightarrow A$
in $\freesmc\scatA$ such that $\Sf\alpha=g$ and $\Sf\alpha'=g'$.
Moreover, a section $\Sf(A')\rightarrow\Sf(A)$ of $\Sf\alpha$ in
$\PSh\scatA$ is necessarily bijective on indices and hence of the
form $\Sf\sigma$ for $\sigma:A'\rightarrow A$ in $\freesmc\scatA$. 
Finally, by
Proposition~\ref{S_Faithful_And_Conservative}(\ref{S_Faithful}), the
identity $\Sf(\sigma\dcomp\alpha)=\id_{\Sf(A')}$ implies that $\sigma$
is a section of $\alpha$. \qed
\end{proof}
\begin{proposition}
Every analytic functor $\PSh\gpdG\rightarrow\PSh\scatC$ with $\gpdG$ a small
groupoid is engendered by its compact generic elements uniquely up to
isomorphism.  
\end{proposition}
\begin{proof}
It is enough to consider $\lan P:\PSh\gpdG\rightarrow\PSh\scatC$ for
$P:\freesmc\gpdG\rightarrow\PSh\scatC$.  In which case, for every 
$(p\ntensor G x)\in \lan P\,X\,c$ one has
$(p\ntensor G\id_{\Sf G})\act_{\lan P}x$. \qed
\end{proof}

Most importantly, generic elements of functors between presheaf categories
over groupoids are invariant under the functorial action.
\begin{lemma}\label{Generic_Invariance_Lemma}
Let $F:\PSh\gpdG\rightarrow\PSh\gpdH$ for $\gpdG$ and $\gpdH$
small groupoids.  If 
$x\in F(\Sf G)(h)$ is generic then so is 
the element
$(\xi \act_F x \act_F \Sf\sigma)\in F(\Sf G')(h')$ for all 
$\sigma:G\rightarrow G'$ in $\freesmc\gpdG$ and 
$\xi: h' \rightarrow h$ in $\gpdH$.
\end{lemma}
\begin{proof}
We first show that $(x\act_F\Sf\sigma)\in F(\Sf G')(h)$ is generic.  So,
consider a cospan $f: \Sf G' \rightarrow Z \leftarrow Y: g$ in $\PSh\gpdG$ and
$y\in F\,Y\,h$ such that $(x\act_F\Sf\sigma)\act_F f = y\act_F g$ in
$F\,Z\,h$.  As $x$ is generic, there exists $k:\Sf G\rightarrow Y$ in
$\PSh\gpdG$ such that $\Sf(\sigma)\act f=k\act g$ and $x\act_F k = y$.  Then,
$(\Sf\sigma^{-1})\act k:\Sf G'\rightarrow Y$ exhibits $x\act_F\Sf\sigma$ as
generic.  

Second, let us see that 
$(\xi\act_F x)\in F(\Sf G)(h')$ is generic.  To
this end, consider a cospan $f: \Sf G\rightarrow Z \leftarrow Y: g$ in
$\PSh\gpdG$ and $y\in F\,Y\,h'$ such that $(\xi\act_F x)\act_F f = y\act_F g$
in $F\,Z\,h'$.  Then, $x\act_F f = (\xi^{-1}\act_F y)\act_F g$ and since $x$
is generic, there exists $k:\Sf G\rightarrow Y$ in $\PSh\gpdG$ such that
$f = k\act g$ and $x\act_F k= \xi^{-1}\act_F y$; so that 
$(\xi\act_F x)\act_F k= y$. \qed
\end{proof}

For $F:\PSh\gpdG\rightarrow\PSh\gpdH$, define
$$\begin{array}{ll}
F\gen(G)(h)
\, \eqdef \,
\big\{\,
  x\in F(\Sf G)(h) \mbox{ \large$\mid$ }
  \mbox{$x$ is generic}
\,\big\}
&\qquad(G\in\freesmc\gpdG, h\in\gpdH^\op)
\end{array}$$
By Lemma~\ref{Generic_Invariance_Lemma}, for $\gpdG$ and $\gpdH$
small groupoids, we have a functor
$F\gen:\freesmc\gpdG\rightarrow\PSh\gpdH$
with 
action, for $\sigma$ in $\freesmc\gpdG$ and $\xi$ in $\gpdH$, given by
$F\gen(\sigma)(\xi) 
\eqdef
F(\Sf\sigma)(\xi) 
$.
As $F\gen$ is a subfunctor of the restriction of $F$ along $\Sf_{\gpdG}$, we
have the following situation 
$$\xymatrix{
\freesmc\gpdG \ar[dr]_-{F\gen} \ar[r]^{\Sf_\gpdG} 
\ar@{}[r]|-{\hspace{4.5mm}\raisebox{-13mm}{$\natrightarrow$}}
& \PSh\gpdG \ar[d]^-F
\\
& \PSh\gpdH
}$$
from which, by the universal property of left Kan extensions, we obtain a
canonical natural transformation
$
\eta^F: 
  \lan{\makebox[5mm]{$F\gen$}} \natrightarrow F:
    \PSh\gpdG\rightarrow\PSh\gpdH
$
explicitly given by 
$$\begin{array}{rcl}
\coend^{G\in\freesmc\gpdG}F\gen(G)(h)\times\freesmc\gpdG\hom{\Sf G,X}
& 
\myarrow{12}{{\eta^F}_{X,h}}{}
& 
F(X)(h)
\\[2mm]
p\ntensor{G}x & \mapsto & p\act_F x
\end{array}$$
These mappings will be now shown to be injective.  To this end, we need
consider an important minimality property of generic elements
(see~\cite[Appendice, D\'efinition~5]{Joyal1234}).

\begin{definition}
For $F:\PSh\scatA\rightarrow\PSh\scatB$, we say that $x\in F\,X\,b$ is
\Def{minimal} if for every $y\in F\,Y\,b$ and $f: Y \rightarrow X$ in
$\PSh\scatA$, $y\act_F f = x$ implies $f$ epi.
\end{definition}

\begin{proposition}
For $P:\freesmc\gpdG\rightarrow\PSh\scatC$ with $\gpdG$ a small groupoid,
$(p\ntensor{G}x)\in\lan P\,X\,c$ is minimal iff $x$ is epi.
\end{proposition}
\begin{proof}
($\Rightarrow$) Follows from the definition of minimality using that
$(p\ntensor G\id_{\Sf G})\act_{\lan P} x=p\ntensor G x$.

($\Leftarrow$) Let $\big(q\ntensor{G'}y\big)\in\lan P\,Y\,c$ and
$f:Y\rightarrow X$ in $\PSh\gpdG$ be such that $q\ntensor{G'}(y\act f) =
\big(q\ntensor{G'} y\big)\act_{\lan P} f = (p\ntensor G x)$.  It follows that
there exists an isomorphism $\sigma: G'\rightarrow G$ in $\freesmc\gpdG$ such
that $(\Sf\sigma)\act x = y\act f$.  Thus, if $x$ is epi then so is $f$. \qed
\end{proof}

\begin{proposition}
The generic elements of a functor between presheaf categories are minimal.
\end{proposition}
\begin{proof}
By Lemma~\ref{Generic=>Minimal}. \qed
\end{proof}

\begin{proposition}
For every $F:\PSh\gpdG\rightarrow\PSh\gpdH$ with $\gpdG$ and $\gpdH$
small groupoids, its associated natural transformation $\eta^F$ is a
monomorphism.  \end{proposition}
\begin{proof}
Let $p\ntensor{G}x$ and $q\ntensor{G'}y$ in
$\lan{\makebox[5mm]{$F\gen$}}X\,h$ be such that $p\act_F x = q\act_F y$.

Since $p\in F(\Sf G)(h)$ is generic and $q\in F(\Sf G')(h)$ is minimal, there
exists an epimorphism $f:\Sf G\rightarrow\Sf G'$ in $\PSh\gpdG$ such that
$p\act_F f = q$ and $f\dcomp y = x$.
Analogously, since $q\in F(\Sf G')(h)$ is generic and $p\in F(\Sf G)(h)$ is
minimal, there exists an epimorphism $g:\Sf G'\rightarrow\Sf G$ in $\PSh\gpdG$
such that $q\act_F g = p$ and $g\dcomp x = y$.

By Proposition~\ref{BijSurInjOnIndices}(\ref{BijSurOnIndices}), $f$ and $g$
are bijective on indices and hence there exist $\sigma:G\rightarrow G'$ and
$\tau:G'\rightarrow G$ in $\freesmc\gpdG$ such that $\Sf\sigma = f$ and $\Sf\tau
= g$.

It follows that
\begin{center}
\hfill$\begin{array}[b]{rclcl}
p\ntensor{G}x 
& = & p\ntensor{G}(f\dcomp y) & = & p\ntensor{G}(\Sf(\sigma)\dcomp y)
\\[2mm]
& = & (p\act_{\lan{\makebox[3mm]{\scriptsize$F\gen$}}}\sigma)\ntensor{G'}y
& = & (p\act_F\Sf\sigma)\ntensor{G'}y
\\[2mm]
& = & (p\act_F f)\ntensor{G'}y
& = & q\ntensor{G'}y
\end{array}$\qed
\end{center}
\end{proof}

Thus, a functor between presheaf categories over groupoids is analytic iff it
is engendered by its compact generic elements. 
\begin{corollary}\label{Analytic=CompactGenericEngendered}
A functor $F:\PSh\gpdG\rightarrow\PSh\gpdH$ with $\gpdG$ and $\gpdH$ small
groupoids is analytic iff its associated natural transformation
$\eta^F:\lan{\makebox[5mm]{$F\gen$}}\natrightarrow F$ is an epimorphism.
\end{corollary}
\hide{
\begin{proof}
  \qed
\end{proof}
}
In particular, the coefficients functor of an analytic functor between
presheaf categories over groupoids is characterised by its generic elements.
Furthermore, since quasi-cartesian natural transformations between such
analytic functors are precisely those that preserve generic elements, this
correspondence extends to an equivalence of categories between coefficient
functors (and natural transformations) 
and analytic functors (and quasi-cartesian natural 
transformations). 
\begin{proposition}
\begin{enumerate}[(i)]
\item
Quasi-cartesian natural transformations between functors
$\PSh\scatA\rightarrow\PSh\scatB$ preserve generic elements. 

\item
If a natural transformation between analytic functors
$\PSh\gpdG\rightarrow\PSh\scatC$ with $\gpdG$ a small groupoid preserves
generic elements then it is quasi-cartesian.   
\end{enumerate}
\end{proposition}
\hide{
\begin{proof}
  \qed
\end{proof}
}

\begin{corollary}\label{EquivalenceCorollary}
For small groupoids $\gpdG$ and $\gpdH$, the functors 
$$\xymatrix{
\CAT\hom{\freesmc\gpdG,\PSh\gpdH} 
\ar[r]<1ex>^-{\lan{(\fstarg)}}
& 
\ar[l]<1ex>^-{(\fstarg)\gen}
\gAF\hom{\gpdG,\gpdH} 
}$$
establish an equivalence of categories.
\end{corollary}

\section{Characterisation of analytic functors}

We conclude the paper with two characterisations of analytic functors
between presheaf categories over groupoids by means of preservation
properties.
As a first step in this direction, we leave the verification of the
following to the reader.
\begin{proposition}
Analytic functors $\PSh\scatA\rightarrow\PSh\scatB$ preserve filtered
colimits.  For $\scatA$ a groupoid, they further preserve wide quasi-pullbacks
and cofiltered limits.  
\end{proposition}
\hide{
\begin{proof}
\qed
\end{proof}
}
Recall that a \Rem{wide quasi-pullback} is a commutative diagram
$\big(\!\xymatrix@C=12.5pt{Q\ar[r]\ar@/^1em/[rr]&D_i\ar[r]&D}\!\big)_{i\in
  I}$ for an indexing set $I$ such that the unique mediating morphism from
the cone
$\big(\!\xymatrix@C=12.5pt{Q\ar[r]\ar@/^1em/[rr]&D_i&D}\!\big)_{i\in I}$
to a limiting cone of the diagram
$\big(\!\xymatrix@C=12.5pt{D_i\ar[r]&D}\!\big)_{i\in I}$  is an
epimorphism.

\begin{corollary}
Analytic endofunctors on presheaf categories over groupoids have both initial
algebra and final coalgebra.
\end{corollary}

We will now consider the following properties of functors between presheaf
categories over groupoids:
\begin{enumerate}[$(1)$]
  \item\label{PresFilColim}
    preservation of filtered colimits,

  \item\label{PresEpi}
    preservation of epimorphisms,

  \item\label{PresQpbk}
    preservation of quasi-pullbacks,

  \item\label{PresWQpbk}
    preservation of wide quasi-pullbacks,

  \item\label{PresCofilLim}
    preservation of cofiltered limits,

  \item\label{EngMin}
    being engendered by compact minimal elements,

  \item\label{EngGen}
    being engendered by compact generic elements (\ie~analytic).
\end{enumerate}
and 
show
\begin{center}\begin{tabular}{ll}
Proposition~\ref{Compact_Minimal_Engendered}: & 
$(\ref{PresFilColim})\thinspace\&\thinspace(\ref{PresEpi})
\Rightarrow
(\ref{EngMin})$
\\[1mm]
Proposition~\ref{Wide_QuasiPullbacks_Engendered}:
& 
$(\ref{PresWQpbk})\thinspace\&\thinspace(\ref{EngMin})
\Rightarrow
(\ref{EngGen})$
\\[1mm]
Proposition~\ref{Cofiltered_Limits_Engendered}:
& 
$(\ref{PresQpbk})\thinspace\&\thinspace(\ref{PresCofilLim})\thinspace\&\thinspace(\ref{EngMin})
\Rightarrow
(\ref{EngGen})$
\end{tabular}\end{center}
so that, since
$(\ref{PresWQpbk}) \Rightarrow (\ref{PresQpbk}) \Rightarrow
(\ref{PresEpi})$, we have that
\begin{center}
  \hfill
  $(\ref{PresFilColim})\thinspace\&\thinspace(\ref{PresWQpbk}) 
  \Rightarrow
  (\ref{EngGen})$
  \quad
  and 
  \quad
  $(\ref{PresFilColim})\thinspace\&\thinspace(\ref{PresQpbk})\thinspace\&\thinspace(\ref{PresCofilLim})
  \Rightarrow
  (\ref{EngGen})$
  \hfill\null
\end{center}

\begin{definition}
A functor $F:\PSh\scatA\rightarrow\PSh\scatB$ is said to be \noRem{engendered
by its (compact) minimal elements} whenever for every $x\in F\,X\,b$ there
exists a minimal element $x_0\in F(X_0)(b)$ (with $X_0=\Sf A$ for
$A\in\freesmc\scatA$) and $f:X_0\rightarrow X$ in $\PSh\scatA$ such that
$x_0\act_F f=x$.
\end{definition}

\begin{proposition}\label{Compact_Minimal_Engendered}
Every functor $\PSh\gpdG\rightarrow\PSh\scatC$, with $\gpdG$ a
small groupoid, preserving filtered colimits and epimorphisms is
engendered by its compact minimal elements.
\end{proposition}
\begin{proof}
Let $F: \PSh\gpdG\rightarrow\PSh\scatC$ be a functor, with $\gpdG$ a
small groupoid, preserving filtered colimits and epimorphisms, and let 
$x\in F\,X\,c$.

Since $X\in\PSh\gpdG$ is a filtered colimit of finitely presentable
objects, 
there exist a finitely presentable object $X_0\in\PSh\gpdG$, an
element $x_0\in F(X_0)(c)$, and a morphism $f: X_0\rightarrow X$ in
$\PSh\gpdG$ such that $x_0\act_F f = x$.

Further, since finitely presentable objects in $\PSh\gpdG$ are
quotients of finite coproducts of representables, 
there exist an object $G\in\freesmc\gpdG$, an element 
$x_1\in F(\Sf G)(c)$, and an epimorphism $q: \Sf G\rightepi X_0$ in
$\PSh\gpdG$ such that $x_1\act_F q = x_0$.

Let 
$G'\in\freesmc\gpdG$, $x'\in F(\Sf G')(c)$, and 
$m: \Sf(G')\rightmono\Sf(G)$ a monomorphism in $\PSh\gpdG$ be such
that $x'\act_F m = x_1$ with $\card{G'}$ chosen minimally. 
We have that $x'$ is a compact element engendering $x$, 
and we now show that it is minimal.

Indeed, consider $y \in F\,Y\,c$ and $g: Y \rightarrow \Sf(G')$ in
$\PSh\gpdG$ such that $y\act_F g = x'$.  Note that the epi-mono
factorisation of $g$ is of the form 
$$\xymatrix{
Y \ar@{->>}[dr]_-\epsilon \ar[rr]^-g && \Sf(G') 
\\
& \Sf(G_0) \ar@{>->}[ur]_-\mu & 
}$$
because, as $\gpdG$ is a groupoid, 
$$\begin{array}{rcll}
\Sub_{\PSh\gpdG}(\Sf G) 
& = & 
\Bigsetof{\
  \sum_{i\in I}\yon(G_i)
  \ \big| \
  I\subseteq\card G
\ }
&\qquad(G\in\freesmc\gpdG)
\end{array}$$
Thus we have $G_0\in\freesmc\gpdG$, 
$y\act_F\epsilon\in F(\Sf(G_0))(c)$, and the monomorphism 
$\mu\dcomp m: \Sf(G_0)\rightmono\Sf(G)$ in $\PSh\gpdG$ satisfying
$(y\act_F\epsilon)\act_F(\mu\dcomp m) = x_1$, from which it follows by
the minimality of $\card{G'}$ that $\card{G'} \subseteq \card{G_0}$.
Hence, 
the monomorphism~$\mu$ is bijective on indices and therefore 
(since $\gpdG$ is a groupoid) an isomorphism, establishing that $g$
is epi. \qed
\end{proof}

\begin{lemma}\label{Generic_Characterisation}
\begin{enumerate}[(i)]
\item\label{Generic_Characterisation_One}
For $F: \PSh\scatA\rightarrow\PSh\scatB$, if $x\in F\,X\,b$ is generic
then for every minimal $y\in F\,Y\,b\,$ and $f: Y \rightarrow X$ in
$\PSh\scatA$, $y\act_F f = x$ implies $f$ iso. 

\item\label{Generic_Characterisation_Two}
Let $F: \PSh\scatA\rightarrow\PSh\scatB$ be a functor engendered by
its (compact) minimal elements and preserving quasi-pullbacks.  For
$x\in F\,X\,b$, if for every (compact) minimal element $y\in F\,Y\,b$ (with
$Y=\Sf A$ for $A\in\freesmc\scatA$) and $f:Y\rightarrow X$ in $\PSh\scatA$,
$y\act_F f = x$ implies $f$ iso, then $x$ is generic.  
\end{enumerate}
\end{lemma}
\begin{proof}
(\ref{Generic_Characterisation_One})~
Assume the hypotheses.
By Lemma~\ref{Generic=>Minimal}, $f$ has a section 
$g: X \rightarrow Y$ in $\PSh\scatA$ such that $g\act_F x = y$.  Since
$y$ is minimal, $g$ is epi and hence an iso, and then so is $f$.

\medskip
(\ref{Generic_Characterisation_Two})~Let $x\in F\,X\,b$ satisfy the hypothesis
of the statement, and let the cospan $f: X\rightarrow Z\leftarrow Y: g$ in
$\PSh\scatA$ and $y\in F\,Y\,b$ be such that $x\act_F f = y\act_F g$ in
$F\,Z\,b$.

Consider a pullback square
$$\xymatrix{
P \ar[r]^-q \ar[d]_-p & Y \ar[d]^-g \\
X \ar[r]_-f & Z
}$$
in $\PSh\scatA$.  Since $F$ preserves quasi-pullbacks, there exists
$z \in F\,P\,b$ such that $z\act_F p = x$ and $z\act_F q = y$.
Further, since $F$ is engendered by its (compact) minimal elements,
there exists $Z_0\in\PSh\scatA$ (with $Z_0 = \Sf A$ for
$A\in\freesmc\scatA$), $z_0\in F(Z_0)(b)$ minimal, and
$h:Z_0\rightarrow Z$ in $\PSh\scatA$ such that $z_0\act_F h = z$.

By hypothesis then, as $z_0\act_F(h\dcomp p) = x$, we have that
$h\dcomp p: Z_0\rightarrow X$ in $\PSh\scatA$ is an isomorphism.
We thus have $(h\dcomp p)^{-1} \dcomp h \dcomp q: X\rightarrow Y$ in
$\PSh\scatA$ such that
$$\begin{array}{rclcl}
\big((h\dcomp p)^{-1} \dcomp h \dcomp q\big) \dcomp g 
& = & (h\dcomp p)^{-1} \dcomp h \dcomp p \dcomp f 
& = & f
\end{array}$$
and 
$$\begin{array}{rclclcl}
x \act_F \big((h\dcomp p)^{-1} \dcomp h \dcomp q\big) 
& = & z_0 \act_F (h \dcomp q) 
& = & z \act_F q 
& = & y
\end{array}$$
showing that $x$ is generic. \qed
\end{proof}

\begin{proposition}\label{Wide_QuasiPullbacks_Engendered}
Every functor $\PSh\gpdG\rightarrow\PSh\scatC$, with $\gpdG$ a
small groupoid, engendered by its compact minimal elements and preserving
wide quasi-pullbacks is engendered by its compact generic elements. 
\end{proposition}
\begin{proof}
Let $F:\PSh\gpdG\rightarrow\PSh\scatC$, with $\gpdG$ a small groupoid, be a
functor engendered by its compact minimal elements and preserving 
wide quasi-pullbacks.

For $x\in F\,X\,b$ consider the wide cospan 
$$\
\nabla 
= 
\bigpair{\
  \nabla_{(x_0,f)} = f:\Sf G\rightarrow X\mbox{ in }\PSh\gpdG
  \ \big| \ 
  x_0 \in F(\Sf G)(b)\mbox{ is minimal and } x_0\act_F f = x
\ }
$$
and let $\pi: P\conerightarrow \nabla$ be a limiting cone in
$\PSh\gpdG$.  (Note that, as $F$ is engendered by its compact minimal
elements, $\nabla$ is non-empty.) 

Since $F$ preserves wide quasi-pullbacks, there exists $p\in F\,P\,b$ such
that, for all minimal ${x_0 \in F(\Sf G)(b)}$ and $f: \Sf G\rightarrow X$ in
$\PSh\gpdG$ with $x_0\act_F f = x$, we have that $p\act_F\pi_{(x_0,f)} = x_0$.
Thus, the cone $\pi$ consists of epimorphims. 

\medskip
We now show the following general property: 
\begin{equation}\label{Mono_Lemma}\begin{minipage}{\minipagelength}
For all minimal $y\in F(\Sf(G'))(b)$ and $g:\Sf(G')\rightarrow P$ in
$\PSh\gpdG$ such that $y\act_F g = p$, it follows that $g$ is split
mono. 
\end{minipage}\end{equation}
Indeed, with respect to any minimal ${x_0 \in F(\Sf G)(b)}$ and 
$f: \Sf G\rightarrow X$ in $\PSh\gpdG$ with $x_0\act_F f = x$, we have the
endomorphism
$$\xymatrix{
\Sf(G') \ar[r]_-g \ar@/^1.25pc/[rrrr]<1ex>^-{e_{(x_0,f)}} & P 
\ar@{->>}[rrr]_-
  {\pi_{\mbox{\scriptsize(}y,g\dcomp\pi_{(x_0,f)}\dcomp f\mbox{\scriptsize)}}} 
&&& \Sf(G')
}$$
(since $y$ is minimal and 
$y\act_F(g\dcomp\pi_{(x_0,f)}\dcomp f)
   = p\act_F (\pi_{(x_0,f)}\dcomp f)
   = x_0\act_F f
   = x$)
satisfying
$$\begin{array}{rclcl}
y\act_F \big(g \dcomp \pi_{(y,g\dcomp\pi_{(x_0,f)}\dcomp f)}\big)
& = & p \act_F \pi_{(y,g\dcomp\pi_{(x_0,f)}\dcomp f)}
& = & y
\end{array}$$
which, by the minimality of $y$, is then an epimorphism.  Thus, 
$e_{(x_0,f)}$ is bijective on indices and, as $\gpdG$ is a groupoid,
an isomorphism; which makes $g$ a split mono.

\medskip
As $F$ is engendered by its compact minimal elements it follows
from~(\ref{Mono_Lemma}) that there exists $G_0\in\freesmc\gpdG$,
$p_0\in F(\Sf(G_0))(b)$, and a section $m:\Sf(G_0)\rightmono P$ in
$\PSh\gpdG$ such that $p_0\act_F m = p$.  Since such a $p_0$ engenders
$x$ (as  
$p_0\act_F (m \dcomp \pi_{(x_0,f)}\dcomp f)
  = p\act_F (\pi_{(x_0,f)}\dcomp f)
  = x_0\act_F f
  = x$), 
we conclude the proof by showing that it further satisfies the
hypothesis of
Lemma~\ref{Generic_Characterisation}(\ref{Generic_Characterisation_Two}).
Indeed, let $y\in F(\Sf(G'))(b)$ be minimal and 
$f: \Sf(G')\rightarrow \Sf(G_0)$ in $\PSh\gpdG$ be such that 
$y\act_F f = p_0$.  Since $p_0$ is minimal, $f$ is epi.  Further,
since $y$ is minimal and $y\act_F(f\dcomp m) = p_0\act_F m = p$, we
have from~(\ref{Mono_Lemma}) that $f\dcomp m$ is split mono.  It follows
that $f$ is split mono, and thus an iso. \qed
\end{proof}

\begin{proposition}\label{Cofiltered_Limits_Engendered}
Every functor $\PSh\gpdG\rightarrow\PSh\scatC$, with $\gpdG$ a
small groupoid, engendered by its compact minimal elements, and preserving
quasi-pullbacks and cofiltered limits is engendered by its compact
generic elements.   
\end{proposition}
\begin{proof}
Let $F$ be a functor as in the hypothesis.

We first show that 
\begin{equation}\label{Well-Founded_Lemma}\begin{minipage}{\minipagelength}
Every infinite cochain
$$\xymatrix{
x_0 & \ar@{->>}[l]_-{g_1} x_1 & \ar@{->>}[l] \cdots &
\ar@{->>}[l]_-{g_i} x_i & \ar@{->>}[l] \cdots
}
\qquad(i\in\Nat)$$
with $x_i\in F(\Sf(G_i))(c)$ minimal and 
$g_i:\Sf(G_{i+1})\rightepi\Sf(G_i)$ in $\PSh\gpdG$ such that
${x_{i+1}\act_F g_{i+1} = x_i}$ for all $i\in\Nat$, stabilises;~\ie~there
exists $i_0\in\Nat$ such that $g_i$ is an iso for all $i\geq i_0$.  
\end{minipage}\end{equation}
Indeed, let 
$$
\begin{array}{c}\xymatrix{
\Sf(G_0) & \ar@{->>}[l]_-{g_1} \Sf(G_1) \ar@{}[rd]|(.45){\cdots} &
\ar@{->>}[l] \cdots & \ar@{->>}[l]_-{g_i} \Sf(G_i) \ar@{}[d]|(.5)\cdots &
\ar@{->>}[l] \cdots 
\\
& \ar[lu]|-{\pi_0} P \ar[u]|-{\pi_1} \ar[rru]|-{\pi_i} & &
}\end{array}
\qquad(i\in\Nat)$$
be limiting in $\PSh\gpdG$.  As $F$ preserves cofiltered limits there
exists $p\in F\,P\,c$  such that $p\act_F\pi_i=x_i$ for all
$i\in\Nat$.
Further, since $F$ is engendered by its compact minimal elements,
there exist $x\in F(\Sf G)(c)$ minimal and $f:\Sf G\rightarrow P$ in
$\PSh\gpdG$ such that $x\act_F f = p$.  Thus, as 
$x\act_F (f\dcomp\pi_i) = x_i$ is minimal, we have epimorphisms
$f\dcomp\pi_i:\Sf(G)\rightepi\Sf(G_i)$ in $\PSh\gpdG$ for all
$i\in\Nat$.  It follows that $\card{G_i} \subseteq \card{G}$ for all
$i\in\Nat$ and hence, since
$\card{G_i}\subseteq\card{G_{i+1}}$~($i\in\Nat$) that there exists
$i_0\in\Nat$ such that $\card{G_i} = \card{G_{i+1}}$ for all $i\geq i_0$.
Thus, for all such $i$, we have that $g_i$ is bijective on indices
and consequently, as $\gpdG$ is a groupoid, an iso. 

Now, for $x\in F\,X\,c$, consider the set $\widewidehat{x}$ of finite
cochains  
$$\xymatrix{
x & \ar[l]_-e x_0 & \ar@{->>}[l]_-{e_1} x_1 & \ar@{->>}[l] \cdots &
\ar@{->>}[l]_-{e_n} x_n & 
}\qquad(n\in\Nat)$$
with $x_i\in F(\Sf(G_i))(c)$ minimal for all $0\leq i\leq n$,
$e:\Sf(G_0)\rightarrow X$ in $\PSh\gpdG$ such that $x_0\act_F e = x$,
and proper epis~(\ie~not isos) $e_i:\Sf(G_{i+1})\rightepi\Sf(G_i)$ in
$\PSh\gpdG$ such that $x_{i+1}\act_F e_i = x_i$ for all 
$1\leq i \leq n$. 
Since $F$ is engendered by its compact minimal elements,
$\widewidehat{x}$ is non-empty.  Further,
by~(\ref{Well-Founded_Lemma}) above, every chain in 
$\widewidehat{x}$ under the prefix order is finite; hence the set of
maximal elements 
of $\widewidehat{x}$ is non-empty.  Finally, since for every maximal
cochain $(x\leftepi x_0\leftepi\cdots\leftepi x_n)$ in
$\widewidehat{x}$, we have that $x_n$ engenders $x$ and satisfies the 
hypothesis of
Lemma~\ref{Generic_Characterisation}(\ref{Generic_Characterisation_Two})
we are done. \qed
\end{proof}

We have thus established the following characterisation result.
\begin{theorem}\label{CharacterisationTheorem}
For a functor between presheaf categories over groupoids the following are
equivalent. 
\begin{enumerate}[(i)]
\item
  The functor is analytic (\ie~engendered by its compact generic elements).

\item
  The functor preserves filtered colimits and wide
  quasi-pullbacks. 

\item
  The functor preserves filtered colimits, quasi-pullbacks, and
  cofiltered limits. 
\end{enumerate} 
\end{theorem}
\hide{
\begin{proof}
Assemble Corollary~\ref{Analytic=CompactGenericEngendered} and
Propositions~\ref{Compact_Minimal_Engendered},~\ref{Wide_QuasiPullbacks_Engendered},~\ref{Cofiltered_Limits_Engendered}.
\qed
\end{proof} 
}

\begin{corollary}\label{AFCorollary}
Small groupoids, analytic functors between their presheaf categories, and
quasi-cartesian natural transformations between them form a 
\mbox{$2$-category} $\AF$.
\end{corollary}
\begin{remark}\label{FinalRemark}
The 
$2$-category
$\AF$ provides $2$-dimensional models of the typed and untyped
lambda calculus and of the typed and untyped differential lambda calculus
(\cf~\cite{FioreFOSSACS,Esp}).
\end{remark}

\section*{References}


\begin{thebibliography}{10}
\expandafter\ifx\csname url\endcsname\relax
  \def\url#1{\texttt{#1}}\fi
\expandafter\ifx\csname urlprefix\endcsname\relax\def\urlprefix{URL }\fi
\expandafter\ifx\csname href\endcsname\relax
  \def\href#1#2{#2} \def\path#1{#1}\fi

\bibitem{BLL}
F.\,Bergeron, G.\,Labelle, P.\,Leroux, Combinatorial species and tree-like
  structures, Vol.~67 of Encyclopedia of Mathematics and its Applications,
  Cambridge University Press, 1997.

\bibitem{CattaniWinskel}
G.\,L.\,Cattani, G.\,Winskel, Profunctors, open maps and bisimulation,
  Mathematical Structures in Computer Science 15~(03) (2005) 553--614.

\bibitem{Fiore}
M.\,Fiore, Exact characterisation of generalised analytic functors between
  groupoids, unpublished note (2003/2004).

\bibitem{FioreFOSSACS}
M.\,Fiore, Mathematical models of computational and combinatorial
  structures, in: Foundations of Software Science and Computation
  Structures (FOSSACS 2005), Vol.~3441 of Lecture Notes in Computer
  Science, Springer-Verlag, 2005, pp.~25--46.

\bibitem{Esp}
M.\,Fiore, N.\,Gambino, M.\,Hyland, G.\,Winskel, The cartesian closed
  bicategory of generalised species of structures, Journal of the London
  Mathematical Society 77~(1) (2008) 203--220.

\bibitem{Girard}
J.-Y.\,Girard, Normal functors, power series and $\lambda$-calculus,
  Annals of Pure and Applied Logic 37 (1988) 129--177.

\bibitem{GirardLL}
J.-Y.\,Girard, Linear logic, Theoretical Computer Science 50 (1987) 1--102.

\bibitem{Grothendieck}
A.\,Grothendieck, Cat\'egories fibr\'ees et descente, in: SGA1, Vol.~224 of
  Lecture Notes in Mathematics, Springer-Verlag (1971).

\bibitem{Hasegawa}
R.\,Hasegawa, Two applications of analytic functors, Theoretical Computer
  Science 272~(1--2) (2002) 113--175.

\bibitem{JoyalAdv}
A.\,Joyal, Une th\'eorie combinatoire des s\'eries formelles, Advances in
  Mathematics 42 (1981) 1--82.

\bibitem{Joyal1234}
A.\,Joyal, Foncteurs analytiques et esp\`eces de structures, in:
  G.\,Labelle, P.\,Leroux (Eds.), Combinatoire \'Enum\'erative, Vol.~1234
  of Lecture Notes in Mathematics, Springer-Verlag, 1986, pp.~126--159.

\bibitem{Taylor}
P.\,Taylor, Quantitative domains, groupoids and linear logic, in: D.\,Pitt
  (Ed.), Category Theory and Computer Science, Vol.~389 of Lecture Notes
  in Computer Science, Springer-Verlag, 1989, pp.~155--181.
\end{thebibliography}

\end{document}